\documentclass[a4paper,10pt]{article}

\usepackage[body={160mm,230mm}]{geometry}
\usepackage{a1}
\usepackage[english]{babel}
\usepackage[latin1]{inputenc} 
\usepackage[dvips]{color}
\usepackage{amsfonts,amssymb}
\usepackage{epsfig}

\def\Z{\mathbb Z}

\title{A 3-Variable Bracket}
\author{Sóstenes Lins\footnote{Author e-mail: sostenes.lins@gmail.com.}
\footnote{The financial support of CNPq \& UFPE, Brazil (process
number 306106/2006)is acknowledged.}
 \\DMat / CCEN / UFPE \& \\
 Brazilian Academy of Science}
{\tiny\date{Version 1: \today}}

\begin{document}

\maketitle

\begin{abstract}
Kauffman's bracket is an invariant of regular isotopy of knots and
links which since its discovery in 1985 it has been used in many
different directions: (a) it implies an easy proof of the invariance
of (in fact, it is equivalent to) the Jones polynomial; (b) it is
the basic ingredient in a completely combinatorial construction for
quantum 3-manifold invariants; (c) by its fundamental character it
plays an important role in some theories in Physics; it has been used
in the context of virtual links; it has connections with many objects other objects
in Mathematics and Physics. I show in
this note that, surprisingly enough, the same idea that produces the bracket can be
slightly modified to produce algebraically stronger regular isotopy and
ambient isotopy invariants living in the quotient ring $R/I$, where
the ring $R$ and the ideal $I$ are:
\begin{center}
$R=\Z[\alpha,\beta,\delta]$,\   $I=<\,p_1, p_2\,>$, with
$p_1=\alpha^2 \delta + 2 \alpha \beta  \delta ^2 -\delta ^2+\beta ^2
\delta
   ,\ p_2=\alpha  \beta  \delta
   ^3+\alpha ^2 \delta ^2+\beta
   ^2 \delta ^2+\alpha  \beta
   \delta -\delta.$
\end{center}
It is easy to prove that any pair of links distinguished by the usual bracket is also distinguishable by the new invariant. The contrary is not necessarily true. However, a explicit example of a pair of knots not distinguished by the bracket and distinguished by this new invariant is an open problem.
\end{abstract}

\section{The brackets $\langle\, D\, \rangle$ and $[\,D\,]$ of a link diagram D}

\begin{definition} Kauffman's bracket maps a link diagram $D$ to
$\langle\ D \rangle \in \Z[\alpha^{-1},\alpha]$ and is characterized by
the following properties:
\begin{center}

(i)\ $\langle \bigcirc  \rangle = 1, \hspace{3mm}$ (ii)\ $\langle D
\cup \bigcirc \rangle = (-\alpha^{-2} - \alpha^2)\langle D \rangle$,
\hspace{3mm} (iii)\ $\langle$ \raisebox{-1mm}
{\includegraphics[width=3.5mm]{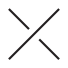}} $\rangle = \
\alpha \langle $ \raisebox{-1mm}
{\includegraphics[width=3.5mm]{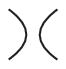}} $\, \rangle +
\alpha^{-1} \langle $ \raisebox{-1mm}{\includegraphics[width=3.5mm]
{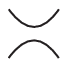}} $\rangle$.
\end{center}
\end{definition}

In this definition, $\bigcirc$  is a diagram of the unknot with no
crossing.  $D \cup \bigcirc$ is a diagram consisting of the diagram
$D$ together with  an extra closed curve $\bigcirc$ that contains no
crossing at all, either with itsef of with $D$. In property (iii)
the three link diagrams are the same except near the crossing where
they are smoothed in the way shown. bserve that the crossing at the
left of (iii) is not invariant under a 90$^{\tiny o}$ rotation and
thus $A$ and $A^{-1}$ at its right can not be interchanged. The
bracket polynomial of a link diagram with $n$ crossings can be
calculated by expressing it as the sum of $2^n$ diagrams with no
crossing by using (iii). By (i) and (ii) it follows that a link
diagram with $k$ components without crossings has $(-\alpha^{-2} -
\alpha^2)^{k-1}$ as its bracket polynomial.
 \label{def:statesum}

As a matter of fact, the bracket is designed to be blind under
Reidemeister move of type II, namely:\ $\langle$
\raisebox{-1mm}{\includegraphics[width=5mm] {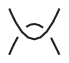}}
$\rangle = \langle$ \raisebox{-1mm}{\includegraphics[width=5mm]
{horizontalSmoothing.eps}}$\rangle$. Let the free variables
$\alpha\,\beta,\,\delta$ satisfy $\langle D \cup \bigcirc \rangle =
\delta \langle D \rangle$ and $\langle$ \raisebox{-1mm}
{\includegraphics[width=3.5mm]{plusCrossing.eps}} $\rangle = \
\alpha \langle $ \raisebox{-1mm}
{\includegraphics[width=3.5mm]{verticalSmoothing.eps}} $\, \rangle +
\beta \langle $ \raisebox{-1mm}{\includegraphics[width=3.5mm]
{horizontalSmoothing.eps}} $\rangle$. We seek restrictions on these
variables to make $\langle$
\raisebox{-1mm}{\includegraphics[width=5mm] {leftMove2.eps}}
$\rangle = \langle$ \raisebox{-1mm}{\includegraphics[width=5mm]
{horizontalSmoothing.eps}}$\rangle$ true. Consider the expansion:
\begin{center}
$\langle $\raisebox{-1mm}{\includegraphics[width=5mm]
{leftMove2.eps}}$\,\rangle$ =
 $\alpha\, \langle$ \raisebox{-1mm}{\includegraphics[width=5mm]
 {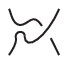}} $\rangle$
 $+\beta\, \langle$
\raisebox{-1mm}{\includegraphics[width=5mm] {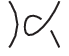}} $\rangle$
= $\alpha^2\, \langle$ \raisebox{-1mm}{\includegraphics[width=5mm]
{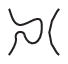}} $\rangle$
 $ + \alpha \beta \langle$
\raisebox{-1mm}{\includegraphics[width=5mm] {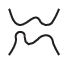}} $\rangle$
+ $\beta \alpha\, \langle $
\raisebox{-1mm}{\includegraphics[width=5mm] {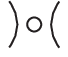}} $\rangle$
 $+\,\beta^2\, \langle$
\raisebox{-1mm}{\includegraphics[width=5mm] {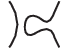}} $\rangle$
= \\ $ \alpha\beta\, \langle$
\raisebox{-1mm}{\includegraphics[width=4mm]
{horizontalSmoothing.eps}} $\rangle$ +
$(\alpha^2+\beta^2+\alpha\beta\delta)\, \langle$
\raisebox{-1mm}{\includegraphics[width=4mm] {verticalSmoothing.eps}}
$\rangle.$
\end{center}
By making (a) $\beta=\alpha^{-1}$ and (b) $\delta=-A^{-2}-A^2$ and
normalizing it to satisfy  $\langle \bigcirc  \rangle = 1$, we get
the bracket, which is insensitive under Reidemeister move II. A
wonderful feature of this scheme is that invariance of the bracket
under Reidemeister move III is for free and it behaves
multiplicatively under Reidemeister moves I: $\langle \hspace{-1mm}$
\raisebox{-1mm}{\includegraphics[width=4mm] {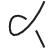}}
$\rangle$ = $-\alpha^3 \langle \hspace{-2mm}$
\raisebox{-1mm}{\includegraphics[height=4mm]
{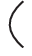}}\,$\rangle$ and $\langle \hspace{-1mm}$
\raisebox{-1mm}{\includegraphics[width=4mm] {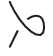}}
$\hspace{-1mm} \rangle$ = $-\alpha^3 \langle \hspace{-2mm}$
\raisebox{-1mm}{\includegraphics[height=4mm]
{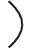}}\,$\rangle$. See Chapter 3 of Lickorish's
book\cite{Lick1997} or the Kauffman's original paper \cite{KA87}.
Because of the simple behavior of $\langle \hspace{3mm} \rangle$
under move I, to obtain an ambient isotopy invariant of links just
define $[\,D\,] = (-\alpha^3)^{-w(D)}\langle\, D\, \rangle.$

\section{The brackets $\langle D \rangle_3$  and $\left[ D \right]_3$ of a link diagram $D$}
My inspiration for this work comes from King \cite{King2005}.
The basic observation is that the above assignments are
too restrictive. Another line of thought provides invariance of
moves II and III under less restrictive assumptions than (a) and
(b). The idea is to consider the exterior of a crossing and fully
expand it. Then we get a more intimate relationship among the
three variables. In fact, after fully expanding the complement of a
pair $m$ of crossings forming the left side of a move II in a link
diagram $D$, there are well defined polynomials
$D^m_1=(D\backslash
m)_1(\alpha,\beta,\delta), \hspace{2mm}D^m_2=(D\backslash
m)_2(\alpha,\beta,\delta) \in \Z[\alpha,\beta,\gamma]$ such that the
exterior is expressible as the sum of two 2-tangles:
\begin{center}
$Ext(m,D)=\langle \hspace{0mm}
\raisebox{-1mm}{\includegraphics[height=4mm] {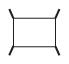}}\,\rangle_3
= D^m_1 \langle \hspace{0mm}
\raisebox{-1mm}{\includegraphics[height=4mm] {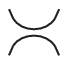}}\,\rangle_3
+ D^m_2 \langle \hspace{0mm}
\raisebox{-1mm}{\includegraphics[height=4mm]
{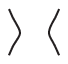}}\,\rangle_3.$
\end{center}
Consider also the two 2-tangles coming from the expansion of $m$
minus the horizontal 2-tangle (the right side of move II):
\begin{center}
$\langle\,m\,\rangle$ - $\langle$
\raisebox{-1mm}{\includegraphics[width=5mm]
{horizontalSmoothing.eps}}$\rangle_3$ = $\langle
$\raisebox{-1mm}{\includegraphics[width=5mm]
{leftMove2.eps}}$\,\rangle$ - $\langle$
\raisebox{-1mm}{\includegraphics[width=5mm]
{horizontalSmoothing.eps}}$\rangle_3$ =$ (\alpha\beta-1)\, \langle$
\raisebox{-1mm}{\includegraphics[width=4mm]
{horizontalSmoothing.eps}} $\rangle_3$ +
$(\alpha^2+\beta^2+\alpha\beta\delta)\, \langle$
\raisebox{-1mm}{\includegraphics[width=4mm] {verticalSmoothing.eps}}
$\rangle_3.$
\end{center}

By matching the corresponding output lines in the product of the
above 2-tangles we get the polynomials which must be zero to make
move II unseen by the new bracket:
\begin{center}
$0=(\alpha\beta-1) D^m_1\, \langle$
\raisebox{-1mm}{\includegraphics[width=8mm]
{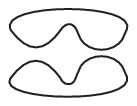}} $\rangle_3$ + $(\alpha\beta-1)
D^m_2\, \langle$ \raisebox{-1mm}{\includegraphics[width=8mm]
{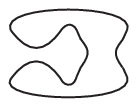}} $\rangle_3$ \\
$+(\alpha^2+\beta^2+\alpha\,\beta\,\delta) D^m_1\, \langle$
\raisebox{-1mm}{\includegraphics[width=8mm]
{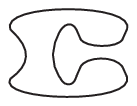}} $\rangle_3$ +
$(\alpha^2+\beta^2+\alpha\beta\delta) D^m_2\, \langle$
\raisebox{-1mm}{\includegraphics[width=8mm]
{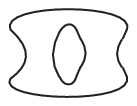}} $\rangle_3 = $\\
\vspace{2mm} $(\alpha\beta-1)\delta^2 D^m_1
+(\alpha\beta-1)\,\delta\, D^m_2+
(\alpha^2+\beta^2+\alpha\beta\delta)\,\delta\, D^m_1 +
(\alpha^2+\beta^2+\alpha\,\beta\,\delta)\,\delta^2\, D^m_2 =$\\
\vspace{1mm}
$[(\alpha\beta-1)\,\delta^2\,+(\alpha^2+\beta^2+\alpha\beta\delta)\,\delta\,]
D^m_1+
[(\alpha\beta-1)\,\delta\,+(\alpha^2+\beta^2+\alpha\,\beta\,\delta)\,\delta^2
]D^m_2.$
\end{center}
Define
$$p_1=(\alpha\beta-1)\,\delta^2\,+(\alpha^2+\beta^2+\alpha\,\beta\,\delta)\,\delta
=\alpha^2 \delta + 2 \alpha\,\beta\,\delta ^2-\delta^2+\beta ^2
\delta,$$ and
$$p_2=(\alpha\beta-1)\,\delta\,+(\alpha^2+\beta^2+\alpha\,\beta\,\delta)\,\delta^2=\alpha
\beta  \delta
   ^3+\alpha ^2 \delta ^2+\beta
   ^2 \delta ^2+\alpha  \beta
   \delta -\delta.$$
If $p_1$ and $p_2$ are zero, then move II is unseen by $\langle\,
\hspace{3mm} \rangle_3$. The algebraic variety defined by the
solution set of the polynomial system of equations consisting of
these two polynomials has 33 branches (easily obtained with
\emph{Mathematica}\texttrademark \cite{mathematica}): \vspace{5mm}

\noindent
$sol_1=\{\beta \to \frac{1}{\alpha },\delta \to
\frac{-\alpha \
^4-1}{\alpha ^2}\},$\\
$sol_2=\{\delta \to -1,\beta \to \alpha
-i\},$\\
$sol_3=\{\delta \to \ -1,\beta \to \alpha +i\},$\\
$sol_4=\{\delta \to
1,\beta \to -\alpha -1\},$\\
$sol_5=\{\delta \to \ 1,\beta \to 1-\alpha
\},$\\
$sol_6=\{\beta \to -\sqrt[6]{-1},\delta \to -1,\alpha \ \to
(-1)^{5/6}\},$\\
$sol_7=\{\beta \to \sqrt[6]{-1},\delta \to -1,\alpha \ \to
-(-1)^{5/6}\},$\\
$sol_8=\{\beta \to -\sqrt[3]{-1},\delta \to 1,\alpha \
\to (-1)^{2/3}\},$\\
$sol_9=\{\beta \to \sqrt[3]{-1},\delta \to 1,\alpha
\to \ -(-1)^{2/3}\},$\\
$sol_{10}=\{\beta \to -i-\sqrt[6]{-1},\delta \to
-1,\alpha \ \to -\sqrt[6]{-1}\},$\\
$sol_{11}=\{\beta \to i-\sqrt[6]{-1},\delta \to \ -1,\alpha \to
-\sqrt[6]{-1}\},$\\
$sol_{12}=\{\beta \to 2 i-\sqrt[6]{-1},\delta \ \to
-1,\alpha \to (-1)^{5/6}\},$\\
$sol_{12}=\{\beta \to -i+\sqrt[6]{-1},\delta \
\to -1,\alpha \to \sqrt[6]{-1}\},$\\
$sol_{13}=\{\beta \to \
i+\sqrt[6]{-1},\delta \to -1,\alpha \to \sqrt[6]{-1}\},$\\
$sol_{14}=\{\beta \ \to -2 i+\sqrt[6]{-1},\delta \to -1,\alpha \to \
-(-1)^{5/6}\},$\\
$sol_{15}=\{\beta \to -1-\sqrt[3]{-1},\delta \to 1,\alpha
\to \ \sqrt[3]{-1}\},$\\
$sol_{16}=\{\beta \to 1-\sqrt[3]{-1},\delta \to
1,\alpha \to \ \sqrt[3]{-1}\},$\\
$sol_{17}=\{\beta \to 2-\sqrt[3]{-1},\delta
\to 1,\alpha \to \ (-1)^{2/3}\},$\\
$sol_{18}=\{\beta \to
-2+\sqrt[3]{-1},\delta \to 1,\alpha \to \ -(-1)^{2/3}\},$\\
$sol_{19}=\{\beta \to -1+\sqrt[3]{-1},\delta \to 1,\alpha \to \
-\sqrt[3]{-1}\},$\\
$sol_{20}=\{\beta \to 1+\sqrt[3]{-1},\delta \to 1,\alpha \
\to -\sqrt[3]{-1}\},$\\
$sol_{21}=\{\delta \to -1,\beta \to -2 i,\alpha \to \
-i\},$\\
$sol_{22}=\{\delta \to -1,\beta \to 2 i,\alpha \to i\},$\\
$sol_{23}=\{\delta \to \ -1,\beta \to i-\sqrt[6]{-1},\alpha \to
-\sqrt[6]{-1}\},$\\
$sol_{24}=\{\delta \ \to -1,\beta \to
-i+\sqrt[6]{-1},\alpha \to \ \sqrt[6]{-1}\},$\\
$sol_{25}=\{\delta \to
-1,\beta \to i-(-1)^{5/6},\alpha \to \ -(-1)^{5/6}\},$\\
$sol_{26}=\{\delta \to -1,\beta \to -i+(-1)^{5/6},\alpha \to \
(-1)^{5/6}\},$\\
$sol_{27}=\{\delta \to 1,\beta \to -2,\alpha \to
1\},$\\
$sol_{28}=\{\delta \to \ 1,\beta \to 2,\alpha \to -1\},$\\
$sol_{29}=\{\delta \to 1,\beta \to \ 1-\sqrt[3]{-1},\alpha \to
\sqrt[3]{-1}\},$\\
$sol_{30}=\{\delta \to 1,\beta \to \ -1+\sqrt[3]{-1},\alpha
\to -\sqrt[3]{-1}\},$\\
$sol_{31}=\{\delta \to 1,\beta \ \to
-1-(-1)^{2/3},\alpha \to (-1)^{2/3}\},$\\
$sol_{32}=\{\delta \to 1,\beta \ \to
1+(-1)^{2/3},\alpha \to -(-1)^{2/3}\},$\\
$sol_{33}=\{\delta \to 0\}.$ \vspace{5mm}

The first of these branches corresponds to the usual bracket. Let
$I$ be the ideal of $R=\Z[\alpha\,,\beta\,,\delta]$ generated by
$p_1,p_2$, that is,  $I = \langle p_1, p_2\rangle$.

\begin{theorem} The class $I + \langle\,
D\, \rangle_3 \in R/I$ is a regular isotopy invariant of the link
diagram $D$.
\end{theorem}
\begin{proof}
The invariance under Reidemeister move II has been established in
the above discussion, because the hypothesis makes the expression to
become zero for arbitrary polynomials $D^m_i(\alpha,\beta,\delta)$,
$i=1,2$. The invariance under move III holds exactly as in the case
of the simple bracket, see \cite{Lick1997}.
\end{proof}

There exists a very simple Gr\"obner basis (\cite{DJD1996}) for the
ideal $I$ in the lexicographic ordering of the monomials with the 3
variables taken in the order $\alpha > \beta > \delta$ is given by
$\mathfrak G=( q_1,\, q_2,\, q_3)$.  The $q$ polynomials are:
\begin{center}
$q_1=\delta ^3 \beta
   ^4-\delta  \beta ^4+\delta
   ^4 \beta ^2-\delta ^2 \beta
   ^2+\delta ^3-\delta,$\\
$q_2=\beta
   \delta ^4+\beta ^3 \delta
   ^3+\alpha  \delta ^3-\beta
   \delta ^2-\beta ^3 \delta
   -\alpha  \delta,$\\
$q_3=\delta
   \alpha ^2+2 \beta  \delta ^2
   \alpha -\delta ^2+\beta ^2
   \delta.$\\
\end{center}
Note that $\langle q_1,\, q_2,\,q_3 \rangle = I =  \langle p_1, p_2
\rangle$. An important property of Gr\"obner basis for an ideal $J$
is that each class $J + p$ has a distinguished representative which
depends on the Gr\"obner basis, on the monomial ordering and on an
ordering of the variables (\cite{DJD1996}). Such a {\em normal form}
or {\em reduced polynomial} is very quick to compute. Henceforth
{\em we define $\langle\, D\, \rangle_3$ to mean this distinguished
element of $I + \langle\, D\, \rangle_3$, modulo $\mathfrak G$ in
the lexicographic monomial ordering relative to the choice $\alpha
> \beta > \delta$.}
The behavior of $\langle \, D\, \rangle_3$ under Reidemeister moves
I are also fairly simple:
\begin{center}
$\langle $\raisebox{-1mm}{\includegraphics[width=6mm]
{plusCurl.eps}}$\,\rangle_3$ = $\langle\,(\beta \delta ^2+\alpha
\delta)$ $ $\hspace{-2mm}
\raisebox{-1mm}{\includegraphics[height=5.5mm]
{openingParenthesis.eps}}$\,\,\rangle_3$ \hspace{3mm} and
\hspace{3mm} $\langle $\raisebox{-1mm}{\includegraphics[width=6mm]
{minusCurl.eps}}$\hspace{-1mm} \rangle_3$ = $\langle\,(\alpha \delta
^2+\beta \delta)$ $ $\hspace{-2mm}
\raisebox{-1mm}{\includegraphics[height=5.5mm]
{closingParenthesis.eps}}$\,\,\,\rangle_3$
\end{center}

\begin{theorem} Assume the link diagram $D$ has non-negative writhe $w$.
Then, $\left[ D \right]_3=(\langle\, \alpha \delta ^2+\beta
\delta)^w D\,\rangle_3$ is an invariant of the ambient isotopy class
of $D$. If the writhe is negative, then $\left[ D \right]_3=
\langle\,(\beta \delta ^2+\alpha \delta)^{-w} D\,\rangle_3$ is an
invariant of the ambient isotopy class of $D$.
\end{theorem}
\begin{proof}
By introducing $|w|$ curls of the opposite sign of $sgn(w)$ we
normalize the links to have writhe zero, and the Theorem follows.
\end{proof}

\begin{conjecture}
There exist pairs of links distinguished by
$\left[ D \right]_3$ and not by $\left[ D \right].$
\end{conjecture}

\end{document}